\documentclass[11pt,a4paper]{article}

\usepackage[left=2.5cm,top=2.5cm,right=2.5cm,bottom=2.5cm]{geometry}
\usepackage{amsmath,amssymb,amsfonts,amsthm}
\usepackage{graphicx}
\usepackage{color}
\usepackage{subfig}
\usepackage{float}                                                                                                               
\usepackage{multirow}
\usepackage{authblk}
\usepackage{hyperref}
\usepackage[ruled,lined,linesnumbered]{algorithm2e}
\usepackage{enumerate}
\usepackage{enumitem}

\newcommand\norm[1]{\lVert #1 \rVert}

\newcommand{\Rn}{\mathbb{R}^n}

\newcommand{\R}{\mathbb{R}}
\usepackage{amssymb}

\theoremstyle{definition}

\newtheorem{obs}{Remark}

\theoremstyle{plain}
\newtheorem{teo}{Theorem}
\newtheorem{lema}{Lemma}

\newtheorem{prop}{Proposition}
\newtheorem{hip}{Assumption}
\usepackage{graphicx}
\begin{document}
	\title{On the regularization property of Levenberg-Marquardt method with Singular Scaling for nonlinear inverse problems \footnote{This work is supported by FAPESC (Fundação de Amparo à Pesquisa e Inovação do Estado de Santa Catarina) [grant number 2023TR000360]}}
	
	\author[1]{Rafaela Filippozzi \footnote{rafaela.filippozzi@posgrad.ufsc.br}}
	\author[1]{Everton Boos\footnote{everton.boos@ufsc.br}}
	\author[1]{Douglas S. Gon\c{c}alves\footnote{douglas.goncalves@ufsc.br}}
	\author[1]{Ferm\'{i}n S. V. Baz\'{a}n\footnote{fermin.bazan@ufsc.br}}
	
	\affil[1]{Department of Mathematics, Federal University of Santa Catarina, Florian\'{o}polis, 88040900, SC, Brazil}
	
	\date{}
	\maketitle

	\begin{abstract}
	Recently, in Applied Mathematics and Computation 474 (2024) 128688, a Levenberg-Marquardt method  (LMM) with Singular Scaling was analyzed and successfully applied in parameter estimation problems in heat conduction where the use of a particular singular scaling matrix (semi-norm regularizer) provided approximate solutions of better quality than those of the classic LMM. Here we propose a regularization framework for the Levenberg–Marquardt method with Singular Scaling (LMMSS) applied to nonlinear inverse problems with noisy data. Assuming that the noise-free problem admits exact solutions (zero-residual case), we consider the LMMSS iteration where the regularization effect is induced by the choice of a possibly singular scaling matrix and an implicit control of the regularization parameter. The discrepancy principle is used to define a stopping index that ensures stability of the computed solutions with respect to data perturbations. Under a new Tangent Cone Condition, we prove that the iterates obtained with noisy data converge to a solution of the unperturbed problem as the noise level tends to zero.  This work represents a first step toward the analysis of regularizing properties of the LMMSS method and extends previous results in the literature on regularizing LM-type methods.
	\end{abstract}
	

	\section{Introduction}
	In this paper, we investigate the global and local convergence properties of the Levenberg–Marquardt method applied to a nonlinear least-squares (NLS) problem of the form:
	\begin{equation}\label{prob1}
		\min_{x\in \mathbb{R}^n} \quad \frac{1}{2}\|F(x)-y\|^2 := \phi(x),
	\end{equation}
	where $F: \mathbb{R}^n \rightarrow \mathbb{R}^m$ is a twice continuously differentiable function. We focus on the \emph{overdetermined} case, where $m \geq n$, and assume that the optimal value of \eqref{prob1} is zero, i.e., the set
	\[
	X^* := \{x \in \mathbb{R}^n : F(x) = y\}
	\]
	is nonempty.
	
	Let  	
	\begin{equation}\label{probdisc}
	F(x)= y, 	
	\end{equation} 	
	be obtained from the discretization of a problem modeling an inverse problem. 
	It is realistic to have noisy data $y^{\delta}$ at disposal, satisfying  	
	\begin{equation}\label{probruido} 		\norm{y-y^{\delta}}\leq \delta, 
	\end{equation} 	
	for some positive $\delta$. Thus, in practice it is necessary to solve a problem of the form 	 	\begin{equation}\label{probdiscruido} 		
		F(x)= y^{\delta}, 	
	\end{equation} 	
	and, due to ill-posedeness, possible solutions may be arbitrarily far from those of the original problem. 
	
	A regularization method is said to be effective if it satisfies the following: when the iteration is stopped at a suitable index $k^*(\delta)$,
	\begin{itemize}
		\item $x^{\delta}_{k^*(\delta)}$ approximates a solution of \eqref{probdisc};
		\item $x^{\delta}_{k^*(\delta)} \to x^*$ as $\delta \to 0$ for some $x^* \in X^*$;
		\item in the noise-free case, convergence to a solution of \eqref{probdisc} occurs.
	\end{itemize}
	In \cite{hanke1997regularizing}, a regularizing version of the Levenberg–Marquardt method was proposed for such inverse problems, under the assumption that the initial guess is sufficiently close to a true solution. This method incorporates an implicit stepsize control and employs the discrepancy principle as a stopping rule:
	\begin{equation}\label{criteriodiscrepancia}
		\|y^\delta - F(x^{\delta}_{k^*(\delta)})\| \leq \tau \delta < \|y^\delta - F(x^\delta_k)\|, \quad 0 \leq k < k^*(\delta),
	\end{equation}
	for some appropriately chosen $\tau > 1$. Hanke showed that the procedure satisfies the regularizing properties listed above, and that local convergence is guaranteed under conditions weaker than the standard local error-bound assumption typically required when the Jacobian $J$ is singular at the target solution \cite{behling2019local,bellavia2016adaptive}.
	
	Inspired by \cite{boos2023levenberg}, throughout this manuscript, we consider the following iteration:
	\begin{align}
		(J_k^T J_k + \lambda_k L^T L) d_k &= -J_k^T(F_k - y^\delta), \label{s1} \\
		x_{k+1} &= x_k + \alpha_k d_k, \quad \forall k \geq 0, \label{s2}
	\end{align}
	where $F_k := F(x_k^\delta)$, $J_k := J(x_k^\delta)$ is the Jacobian of $F$ at $x_k^\delta$, $\lambda_k > 0$ is a regularization parameter, $\alpha_k$ is the step size, and $L^T L$ is a scaling matrix, which is allowed to be singular. 
	We refer to iteration \eqref{s1}--\eqref{s2} as the Levenberg–Marquardt method with Singular Scaling (LMMSS). Here we focus on the “pure” version of this method, i.e., with $\alpha_k = 1$ for all $k$.

	In this work, we prove that the iterates generated for LMSS with noisy data converge to a solution of the unperturbed problem as the noise level tends to zero. This study represents a first step toward the analysis of the regularizing properties of the LMMSS method and extends previous results on regularizing LM-type methods \cite{ bellavia2016adaptive, engl,hanke1997regularizing,Hansen1998,Kaltenbacher2008,Morozov1984}.
	For that, we introduce a variant of the Tangent Cone Condition, adapted to the specific structure of the LMMSS method. %Although we are unable to analytically verify that the assumption is satisfied for the test problems, we provide numerical evidence supporting its validity in all presents cases.
	
	Our study is organized as follows. The assumptions and auxiliary results required for the analysis are presented in Section \ref{sec:assump}. The main theoretical results, including the convergence of the method for noisy data, are established in Section \ref{sec:princiresult}. In particular, we prove that the iterates produced by the method satisfy a regularizing property, as stated in Theorem~\ref{teo:convergenciacomruido}. %Illustrative examples are provided in Section \ref{sec:expnum}, and 
	Final remarks are given in Section \ref{sec:conclusion}.

	\section{Assumptions and Auxiliary Results}\label{sec:assump}
	
	Recall that throughout this manuscript, we assume that the set of stationary points of $\phi$ is non-empty: $X^* \ne \emptyset$. 
	
	Assumption~\ref{Hip_nullJLnew} below guarantees that $d_k$ solution of \eqref{s1} is unique.
	
	\begin{hip}\label{Hip_nullJLnew}
		The matrix $L\in \R^{p\times n}$ has $\text{rank}(L) = p \leq n \leq m$ and there exists $\Omega \subset \Rn$ and $\gamma >0$ such that, for every $x \in \Omega$
		\begin{equation}\label{comp-condition}
			\norm{J(x) v}^2 + \norm{L v} ^2 \geq \gamma \norm{v}^2, \quad \forall v \in \R^n.
		\end{equation}
	\end{hip}
	
	Such assumption comes from the inverse problems literature \cite{Morozov1984,engl} where it is sometimes called \emph{completeness condition}. 
	It was a key assumption in \cite{boos2023levenberg} to study the local convergence of LMMSS in the zero residue setting. 
	It is worth to point out that such assumption is equivalent to 
	\begin{equation*}%\label{kernelJL}
		\mathcal{N} (J(x)) \cap \mathcal{N} (L) = \{ {\bf 0} \}, \quad  \forall x \in \Omega.
	\end{equation*}
	
	We start by mentioning that \eqref{s1} always has a solution. In fact, \eqref{s1} corresponds to the normal equation $B^T B d = B^T c$ of the augmented system
	\[
	Bd := \left[ \begin{array}{c}
		J_k \\
		\sqrt{\lambda_k} L
	\end{array} \right] d = - \left[ \begin{array}{c}
		F_k-y^\delta \\
		0
	\end{array} \right] =: c,
	\]
	and the normal equation always has a solution because $B^T c$ lies in the range of $B^T B$. 
	But it is Assumption~\ref{Hip_nullJLnew} that ensures the solution of \eqref{s1} is unique. 
	Indeed, it is easy to see that condition \eqref{comp-condition} at $x_k$ implies that $J_k^T J_k + \lambda_k L^T L$ is positive definite (and thus nonsingular).

	Now, we will show the Generalized Singular Value Decomposition (GSVD), 
	an important tool in theoretical analysis of the LMMSS direction. 
	
	\begin{teo}(GSVD) \cite[p. 22]{Hansen1998}\label{teoGSVD}
		Consider the pair $(A, L)$, where $A \in \mathbb{R}^{m \times n}$, $L \in \mathbb{R}^{p \times n}$, $m \geq n \geq p$, $rank(L) = p$, and $\mathcal{N} (A) \cap \mathcal{N}(L) = \{ 0 \}$. 
		Then, there exist matrices $U \in \mathbb{R}^{m \times n}$ and $V \in \mathbb{R}^{p \times p}$ with orthonormal columns 
		% (i.e., $U^{ T}U = I_n$ and $V^{ T}V = I_{p }$) 
		and a nonsingular matrix $X \in \mathbb{R}^{n \times n}$ such that 
		\begin{equation}\label{GSVD}
			A = U \left[ \begin{array}{cc}
				\Sigma & 0 \\ 
				0 & I_{n-p}
			\end{array}  \right] X^{-1}  \quad \text{and} \quad L = V \left[ \begin{array}{cc}
				M & 0
			\end{array}  \right] X^{-1},
		\end{equation}
		with $\Sigma$ and $M$ being the following diagonal matrices:
		\begin{equation}
			\Sigma = \text{diag}(\sigma_{1}, \dots, \sigma_{p}) \in \mathbb{R}^{p \times p} \quad \text{and} \quad M =  \text{diag}(\mu_{1}, \dots, \mu_{p}) \in \mathbb{R}^{p \times p}.
		\end{equation}
		Moreover, the elements of $\Sigma$ and $M$ are non-negative, ordered as follows:
		\begin{equation}
			0 \leq \sigma_{1} \leq \dots \leq \sigma_{p} \leq 1 \quad \text{and} \quad 1 \geq \mu_{1} \geq \dots \geq \mu_{p} >0,
		\end{equation}
		and normalized by the relation $\sigma_{i}^{ 2} + \mu_{i}^{ 2} = 1$, for $i = 1, \dots, p$. 
		We call the generalized singular value of the pair $(A, L)$ the ratio
		\begin{equation*}
			\zeta_{i} = \frac{\sigma_{i}}{\mu_{i}}, \quad i = 1, \dots, p.
		\end{equation*}
	\end{teo}

	By considering the GSVD of the pair $(J_k, L)$ we can provide a useful characterization of the direction $d_k$ in LMMSS.
	In fact, given the GSVD 
	
	\begin{equation*}
		J_k = U_k \left[ \begin{array}{cc}
			\Sigma_k & 0 \\
			0 &I_{n-p}
		\end{array}  \right] X_k^{-1}  \quad \text{and} \quad L = V_k \left[ \begin{array}{cc}
			M_k & 0
		\end{array}  \right] X_k^{-1},
	\end{equation*}
	where   \begin{equation}\label{eq:sigmamuigsvd}
		(\Sigma_k)_{ii} := \sigma_{i,k} \quad \text{and} \quad (M_k)_{ii} := \mu_{i,k}, \quad i = 1, \dots, p;
	\end{equation}
	it follows that   
	\begin{equation}\label{JmuL}
		J_k^TJ_k + \lambda_k L^TL = X_k^{-T} \left[ \begin{array}{cc}
			\Sigma_k^2 + \lambda_k M_k^2 & 0 \\ 
			0 & I_{n-p}
		\end{array}  \right] X_k^{-1}.  %\quad \forall k \in \mathbb{N},
	\end{equation}
	
	Then, $d_k$ from \eqref{s1} can be expressed as: 
	\begin{equation}\label{eq:dkGSVD}
		d_k = -X_k \left[ \begin{array}{cc}
			\Gamma_k & 0 \\ 
			0 & I_{n-p}
		\end{array}  \right] X_k^T {J_k}^T (F_k-y^\delta), %\quad \forall k \in \mathbb{N},
	\end{equation}
	with $\Gamma_k:=(\Sigma_k^2 + \lambda_k M_k^2)^{-1}.$

	In \cite{hanke1997regularizing}, the analysis of the regularizing properties of the Levenberg–Marquardt method was carried out under the assumption that problem \eqref{prob1} is solvable and that the so-called \emph{tangential cone condition} holds, which controls the Taylor remainder of $F$.
	
	In our context, however, we work with a variant of the method that incorporates a possibly singular matrix $L$, the Levenberg–Marquardt method with Singular Scaling (LMMSS). This introduces new challenges in the analysis, particularly in controlling the iterates in a geometry induced by $L$, and demands a slightly stronger version of the classical tangential cone condition.
	
	We define the inner product
	\[
	\langle x, y \rangle_{L^T L} := \langle x, L^T L y \rangle,
	\]
	with the corresponding norm
	\[
	\|x\|_L^2 := \langle x, L^T L x \rangle = \langle Lx, Lx \rangle = \|Lx\|_2^2.
	\]
	
	Note that \begin{equation}\label{eq:isolandoLtL}
		\lambda_kL^TL(x_{k+1}-x_{k})= -J_k^T(F_k-y^\delta)-J_k^TJ_k(x_{k+1}-x_{k})
	\end{equation}
	Then, 
	\begin{align}
		\norm{x_{k+1}-x_k}_{L}^2&= (x_{k+1}-x_k)^TL^TL(x_{k+1}-x_k) \nonumber\\
		&=-\frac{1}{\lambda_k}(x_{k+1}-x_k)^TJ_k^T(F_k-y^\delta) -\frac{1}{\lambda_k}\norm{J_k(x_{k+1}-x_k)}^2.
		\label{eq:normLxk}
	\end{align}

	We also define a ball centered at a point \(x \in \mathbb{R}^n\), with radius \(\rho > 0\), in the \(L\)-norm as
	\[
	B_L(x, \rho) := \{ y \in \mathbb{R}^n : \|y - x\|_L \leq \rho \}.
	\]
	
	From now on, we assume that \(x^* \in X^*\) is a solution that minimizes the \(L\)-norm distance to the initial guess \(x_0\), that is,
	\[
	\|x_0 - x^*\|_L \leq \|x_0 - z\|_L, \quad \forall z \in X^*.
	\]
	
	To ensure convergence in this modified setting, we introduce the following assumption, which we will call the TCC-$L$ condition:
	
	\begin{hip}\label{hip:tccLTL}
		Given an initial guess \(x_0\), there exist positive constants \(\rho\) and \(c\) such that the discretized problem \eqref{prob1} is solvable in the ball \(B_L(x_0, 2\rho)\), and
		\begin{equation}\label{eq:tcc}
			\|J(x)(\tilde{x} - x) - F(\tilde{x}) + F(x)\| \leq c \|\tilde{x} - x\|_L \|F(\tilde{x}) - F(x)\|, \quad \forall x, \tilde{x} \in B_L(x_0, \rho).
		\end{equation}
	\end{hip}
	
	Note that when $L= I$ or $L$ is nonsingular , this stronger version of the tangential cone condition reduces to the classical one. Indeed, in this case, the $L$-norm is equivalent to the Euclidean norm, since the singular values of $L$ are all strictly positive.
	
	More precisely, let $\sigma_{\min}$ and $\sigma_{\max}$ denote the smallest and largest singular values of $L$, respectively. Then, for all $x \in \mathbb{R}^n$, we have
	$$
	\sigma_{\min} \|x\| \leq \|Lx\| = \|x\|_L \leq \sigma_{\max} \|x\|.
	$$
	Therefore, the $L$-norm is equivalent to the Euclidean norm and, in particular, we obtain that
	$$
	\|\tilde{x} - x\|_L \|F(\tilde{x}) - F(x)\| \leq \sigma_{\max} \|\tilde{x} - x\| \|F(\tilde{x}) - F(x)\|.
	$$
	Thus, condition \eqref{eq:tcc} implies the classical tangential cone condition:
	$$
	\|J(x)(\tilde{x} - x) - F(\tilde{x}) + F(x)\| \leq c' \|\tilde{x} - x\| \|F(\tilde{x}) - F(x)\|,
	$$
	where $c' := c \, \sigma_{\max}$. This shows that the TCC-$L$ condition is stronger than, and reduces to, the classical TCC when $L$ is nonsingular.

	\begin{hip}\label{assumption2.2bellavia2016}
		Let $x_0,c$ and $\rho$ as in Assumption~\ref{hip:tccLTL}, $x^*$ be a solution of \eqref{probdisc} and $x_0$ satisfy
		\begin{equation}
			\norm{x_0-x^*}_{L}< \min\{\frac{q}{c}, \rho\}, \text{ if }\delta=0,
		\end{equation}
		\begin{equation}
			\norm{x_0-x^*}_{L}< \min\{\frac{q\tau -1}{c(1+\tau)}, \rho\}, \text{ if }\delta>0,
		\end{equation}
		where $\tau>1/q$.
	\end{hip}
	
	This assumption above plays a crucial role in local convergence analysis and has been extensively employed in the literature see,\cite{bellavia2016adaptive,hanke1997regularizing,Hansen1998,Kaltenbacher2008,Morozov1984}.
	
				\section{Regularizing for LMMSS}\label{sec:princiresult}
				
				In the regularizing Levenberg–Marquardt method \cite{hanke1997regularizing}, the parameter $\lambda_k$ is selected as the solution $\lambda_k^q$ of the following nonlinear scalar equation, in order to approximate solutions of problem \eqref{prob1}
				\begin{equation}\label{eq:qconditionigualdade}
					\norm{F_k-y^\delta+J_kd(\lambda)}= q\norm{F_k-y^\delta},
				\end{equation}
				for some fixe $q\in (0,1)$.
				
				Under suitable assumptions discussed below, $\lambda_k^q$ is uniquely determined from \eqref{eq:qconditionigualdade}.
				
				\begin{lema}\label{lem:intervaloqcondition}
					Suppose $\norm{J_k^TF_k} \neq 0$. Let $d_k:= d(\lambda)$ be the solution of \eqref{s1}, with $\lambda>0$ and $P_k$ be the orthogonal projector onto $\mathcal{R}(J_k)^\perp$. Then,
					\begin{enumerate}[label=(\alph*)]
						\item Equation \eqref{eq:qconditionigualdade} is not solvable if $\norm{P_k(F_k-y^\delta)}> q \norm{F_k-y^\delta}.$
						\item If 
						\begin{equation}\label{eq:lemaexistqcondition}
							\norm{F_k-y^\delta +J_k(x^*-x_k^\delta)}\leq \dfrac{q}{\theta}\norm{F_k-y^\delta},	
						\end{equation}
						
						for some $\theta>1$, and $x_*$ is a solution of \eqref{probdisc}, then Equation \eqref{eq:qconditionigualdade} has a unique solution $\lambda_k^q$ such that 
						\begin{equation}\label{eq:intervalolambda}
							\lambda_k^q \in \Big(0, \frac{q}{1-q}\zeta_{p,k}^2\Big),
						\end{equation}
						where $\zeta_{p,k}^2=\frac{\sigma^2_{p,k}}{\mu^2_{p,k}}$.
					\end{enumerate}
				\end{lema}
				\begin{proof}
					\begin{enumerate}[label=(\alph*)]
						\item 
						
						Remember that
						\begin{equation*}
							J_k= U_k \left[ \begin{array}{cc}
								\Sigma_k & 0 \\
								0 &I_{n-p}
							\end{array}  \right] X_k^{-1}
						\end{equation*}
						and 
						\begin{equation*}
							d_k=	-X_k \left[ \begin{array}{cc}
								(\Sigma_k^2 + \lambda_k M_k^2)^{-1} & 0 \\ 
								0 & I_{n-p}
							\end{array}  \right] X_k^T {J_k}^T (F_k-y^\delta)
						\end{equation*} 
						
						and consequently 
						\begin{equation*}
							F_k-y^\delta+J_kd(\lambda)= F_k-y^\delta  - U_k \left[ \begin{array}{cc}
								\Sigma_k^2(\Sigma_k^2 + \lambda M_k^2)^{-1} & 0 \\
								0 &I_{n-p}
							\end{array}  \right]U_k^T(F_k-y^\delta).  
						\end{equation*}
						
						\begin{align}
							\norm{F_k-y^\delta+J_kd(\lambda)}&= \norm{U_k\Big(I-\left[ \begin{array}{cc}
									\Sigma_k^2(\Sigma_k^2 + \lambda M_k^2)^{-1} & 0 \\
									0 &I_{n-p}
								\end{array}  \right]\Big)U_k^T(F_k-y^\delta)}\nonumber\\
							&=\norm{\left[ \begin{array}{cc}
									\Psi_k & 0 \\
									0 &I_{n-p}
								\end{array}\right] U_k^T(F_k-y^\delta)},\label{eq:normFKJF}
						\end{align}
						where $$(\Psi_k)_{ii}= \frac{\lambda \mu_{i,k}^2}{\sigma_{i,k}^2+\lambda \mu_{i,k}^2}.$$
						
						Lets $r= (r_1,\dots, r_n)^T= U_k^T(F_k-y^\delta)$, we have 
						\begin{align}\label{eq:dlambdasum}
							\norm{F_k-y^\delta+J_kd(\lambda)}^2&= \sum_{i=1}^{p}\frac{\lambda^2 \mu_{i,k}^4}{(\sigma_{i,k}^2+\lambda \mu_{i,k}^2)^2}r_i^2+ \sum_{p+1}^n r_i^2
						\end{align}
						Define $\Omega(d_\lambda)=\norm{F_k-y^\delta+J_kd(\lambda)}^2$, then
						$$\Omega'(d_\lambda)= \frac{1}{\Omega(d(\lambda))}\sum_{i=1}^p \frac{2\lambda\mu_{i,k}^2\sigma_{i,k}^2r_i^2}{(\sigma_{i,k}^2+\lambda\mu_{i,k}^2)^3}> 0,$$
						thus,$\Omega(d_\lambda)$ is monotonically increasing as a function of $\lambda$, if 	$\norm{P_k(F_k-y^\delta)}> q \norm{F_k-y^\delta},$ \eqref{eq:qconditionigualdade} does not admit solution.
						
						\item Note that $\norm{P_k(F_k-y^\delta)}=\norm{P_k(F_k-y^\delta+J_k(x-x_k^\delta))} \leq \norm{F_k-y^\delta+J_k(x-x_k^\delta)},$for any $x$. For $x=x^*$ we have, for some $\theta>1$
						\[ \norm{P_k(F_k-y^\delta)}\leq \norm{F_k-y^\delta+J_k(x^*-x_k)} \leq \frac{q}{\theta}\norm{F_k-y^\delta}\leq q\norm{F_k-y^\delta},\]
						then, by (a) \eqref{eq:qconditionigualdade} admits a solution which is positive and unique. 
						
						As$(\Psi_k)_{ii}^{-1}= \frac{\sigma_{i,k}^2+\lambda \mu_{i,k}^2}{\lambda \mu_{i,k}^2} \leq \frac{\sigma_{p,k}^2}{\lambda\mu_{p,k}^2}+1$ by \eqref{eq:normFKJF} 
						\begin{align*}
							\norm{F_k-y^\delta+J_kd(\lambda)}&\geq\frac{\norm{U_k^T(F_k-y^\delta)}}{\norm{\left[ \begin{array}{cc}
										\Psi_k & 0 \\
										0 &I_{n-p}
									\end{array}\right]^{-1}}}\\
							&\geq \frac{\lambda}{\frac{\sigma^2_{p,k}}{\mu^2_{p,k}}+\lambda}\norm{F_k-y^\delta}
						\end{align*}
						
						From \eqref{eq:qconditionigualdade}
						
						$$q\norm{F_k-y^\delta}\geq  \frac{\lambda^q}{\zeta_{p,k}^2+\lambda^q}\norm{F_k-y^\delta},$$
						
						which yields \eqref{eq:intervalolambda} 
						with 
						$\zeta_{p,k} = \frac{\sigma_{p,k}}{\mu_{p,k}}.$
						
						Then, $$q \zeta_{p,k}^2+ q\lambda_k^q \geq \lambda_k^q$$ that implies
						
						$$\lambda_k^q \in \Big(0, \frac{q}{1-q}\frac{\sigma^2_{p,k}}{\mu^2_{p,k}}\Big)$$
					\end{enumerate}
				\end{proof}

			\begin{obs}
				In \cite[Remark, p.~6]{hanke1997regularizing}, it was noted that condition \eqref{eq:qconditionigualdade} can be replaced by the inequality
				\begin{equation}\label{eq:qconditiondesigualdade}
					\|F(x_k^\delta) - y^\delta + J(x_k^\delta)d(\lambda)\| \geq q \|F(x_k^\delta) - y^\delta\|.
				\end{equation}
			 Given that equation \eqref{eq:qconditionigualdade} might not admit a solution, and motivated by the goal of achieving global convergence while retaining regularization properties. Although condition \eqref{eq:qconditiondesigualdade} is expressed as an inequality rather than an equality, it can still be shown that $\lambda_k^q$ belongs to the same interval characterized in \eqref{eq:intervalolambda}. This result follows directly from Lemma~\ref{lem:intervaloqcondition}.
			\end{obs}

				\begin{prop}\label{prop:contilambda}
				Assume that \eqref{eq:lemaexistqcondition} is fulfilled for some $\theta > 1$ and $x^{*}$ being a solution of \eqref{probdisc}. Let $x_{k+1} = x_k + d_k$ with $d_k = d(\lambda_k)$ satisfying \eqref{s1} and \eqref{eq:qconditionigualdade}. Assumption~\ref{Hip_nullJLnew} is hold. Then				\(\lambda_k^q\) is continuous with respect to \(y^\delta\).
				\end{prop}
			\begin{proof}
				We aim to show that the regularization parameter \(\lambda_k^q = \lambda_k^q(y^\delta)\), defined implicitly by the scalar equation
				\begin{equation}\label{eq:qconditionfunction}
					\omega(\lambda, y^\delta) := \left\|F_k - y^\delta + J_k d(\lambda)\right\|^2 - q^2 \left\|F_k - y^\delta\right\|^2,
				\end{equation}
				depends continuously on the noisy data $y^\delta$.
				
				Assume that $F_k$ and $J_k$ are fixed for the current iteration \(k\), and that \(d(\lambda)\) solves \eqref{s1}.
				
				Using the GSVD of the pair \((J_k, L)\), we can write
				\begin{equation}\label{eq:dkGSVD}
					d(\lambda) = -X_k
					\begin{bmatrix}
						\Gamma_k(\lambda) & 0 \\
						0 & I_{n-p}
					\end{bmatrix}
					X_k^\top J_k^\top (F_k - y^\delta),
				\end{equation}
				with \(\Gamma_k(\lambda) := (\Sigma_k^2 + \lambda M_k^2)^{-1}\).
				
				From \eqref{eq:dkGSVD}, we see that \(d(\lambda)\) depends affinely on \(y^\delta\), and smoothly on \(\lambda\) for \(\lambda > 0\). Therefore, \(\omega(\lambda, y^\delta)\) defined in \eqref{eq:qconditionfunction} is continuously differentiable in a neighborhood of any point \((\lambda_0, y^\delta_0)\) satisfying the equation.
				
				Suppose that \(\lambda_k^q\) satisfies \(\omega(\lambda_k^q, y^\delta) = 0\), and that
				\[
				\frac{\partial \omega}{\partial \lambda}(\lambda_k^q, y^\delta) \neq 0.
				\]
				Then, by the Implicit Function Theorem, \(\lambda_k^q\) can be locally expressed as a continuously differentiable function of \(y^\delta\).
				
				To justify the condition \(\frac{\partial \omega}{\partial \lambda} \neq 0\), we differentiate using \eqref{eq:dkGSVD}:
				\[
				\frac{\partial \omega}{\partial \lambda}(\lambda, y^\delta)
				= \frac{1}{\left\|F_k - y^\delta + J_k d(\lambda)\right\|^2}
				\sum_{i=1}^p \frac{2\lambda\mu_{i,k}^2\sigma_{i,k}^2r_i^2}{(\sigma_{i,k}^2+\lambda\mu_{i,k}^2)^3}> 0,
				\]
				The strictness of the inequality above follows from Assumption~\ref{Hip_nullJLnew}, which ensures that at least one $\sigma_i>0$.
				 This implies that \(\omega(\lambda, y^\delta)\) is strictly increasing in \(\lambda\), ensuring the validity of the Implicit Function Theorem.
				
				Hence, we conclude that \(\lambda_k^q = \lambda_k^q(y^\delta)\) depends continuously on the data \(y^\delta\), as claimed.
			\end{proof}
		\begin{obs}\label{obs:xkcontydelta}
			 Note that the continuity of $\lambda_k$ is well defined because of Proposition~\ref{prop:contilambda}, then $x_{k(\delta_n)}^{\delta_n}$ is the result of a combination of continuous operations. 
		\end{obs}
		\iffalse
		\begin{teo}[Implicit Function Theorem]
			Let \(\omega: \mathbb{R}_{>0} \times \mathbb{R}^m \to \mathbb{R}\) be a continuously differentiable function. Suppose there exists a point \((\lambda_0, y_0^\delta)\) such that
			\[
			\omega(\lambda_0, y_0^\delta) = 0 \quad \text{and} \quad \frac{\partial \omega}{\partial \lambda}(\lambda_0, y_0^\delta) \neq 0.
			\]
			Then, there exist open neighborhoods \(U \subset \mathbb{R}^m\) of \(y_0^\delta\) and \(V \subset \mathbb{R}_{>0}\) of \(\lambda_0\), and a unique continuously differentiable function \(\lambda: U \to V\) such that
			\[
			\omega(\lambda(y^\delta), y^\delta) = 0 \quad \text{for all } y^\delta \in U.
			\]
			In particular, \(\lambda(y^\delta)\) depends continuously on the data \(y^\delta\).
		\end{teo}}\fi

		In what follows, we estimate the ``gain'' $\|x_k - x^*\|_{L}^2 - \|x_{k+1} - x^*\|_{L}^2.$ As we are working within the LMMSS framework, the norm induced by $L$ must be employed in the analysis. This estimate is a important result for the theoretical developments presented in this manuscript.
				
				\begin{lema}\label{lema:GANHO}
					Assume that \eqref{eq:lemaexistqcondition} is fulfilled for some $\theta > 1$ and $x^{*}$ being a solution of \eqref{probdisc}. Let $x_{k+1} = x_k + d_k$ with $d_k = d(\lambda_k)$ satisfying \eqref{s1} and \eqref{eq:qconditionigualdade}. Then it holds
					\begin{equation}\label{eq:ganho}
						\norm{ x_k - x^* }_{L}^2 - \norm{ x_{k+1} - x^* }_{L}^2 \geq \norm{ x_{k+1} - x_k }^2_{L}    
					\end{equation}
					
					\begin{align}\label{eq:plusganho}
						\left\| x_k - x^* \right\|_{L}^2 - \left\| x_{k+1} - x^* \right\|_{L}^2 &\geq \frac{2(\theta - 1)}{\theta \lambda_k} \left\| F_k-y^\delta + J_k d_k \right\|_2^2 \\ \label{eq:plusganho2} &\geq  \frac{2(\theta - 1)(1-q)q}{\zeta_{p,k}^2\theta} \left\| F_k-y^\delta \right\|_2^2,
					\end{align}
					where
					$\zeta_{p,k} = \frac{\sigma_{p,k}}{\mu_{p,k}}$.
				\end{lema}
				\begin{proof}
					\begin{align*}
						\norm{x_{k+1}-x^*}_{L}^2-\norm{x_{k}-x^*}_{L}^2&= \norm{x_{k+1}-x_k}_{L}^2 +2\langle x_{k+1}-x_k,x_{k}-x^*\rangle_{L^TL} \\ &=    \norm{L(x_{k+1}-x_k)}^2 +2\langle x_{k+1}-x_k,L^TL(x_{k}-x^*)\rangle \\ &= \norm{L(x_{k+1}-x_k)}^2 +2(L^TL(x_{k+1}-x_k))^T(x_{k}-x^*)
					\end{align*}
					From, \eqref{eq:isolandoLtL} and \eqref{eq:normLxk}
					
					\begin{align*}
						&\norm{x_{k+1}-x^*}_{L}^2-\norm{x_{k}-x^*}_{L}^2\\
						&=-\frac{1}{\lambda_k}(x_{k+1}-x_k)^TJ_k^TF_k -\frac{1}{\lambda_k}\norm{J_k(x_{k+1}-x_k)}^2+ \frac{2}{\lambda_k}[-J_k^TF_k-J_k^TJ_k(x_{k+1}-x_{k})]^T(x_{k}-x^*)\\
						&= -\frac{1}{\lambda_k}(x_{k+1}-x_k)^TJ_k^TF_k -\frac{1}{\lambda_k}\norm{J_k(x_{k+1}-x_k)}^2+ \frac{2}{\lambda_k}[F_k-y^\delta+J_k(x_{k+1}-x_{k})]^TJ_k(-x_{k}+x^*)\\ 
						&=  -\frac{1}{\lambda_k}(x_{k+1}-x_k)^TJ_k^TF_k -\frac{1}{\lambda_k}\norm{J_k(x_{k+1}-x_k)}^2 \\
						&\hspace{4cm} +\frac{2}{\lambda_k}[F_k-y^\delta+J_k(x_{k+1}-x_{k})]^T[F_k-y^\delta+ J_k(-x_{k}+x^*)- F_k-y^\delta]\\ 
						&=-\frac{1}{\lambda_k}(x_{k+1}-x_k)^TJ_k^TF_k -\frac{1}{\lambda_k}\norm{J_k(x_{k+1}-x_k)}^2 \\
						&+\frac{2}{\lambda_k}[F_k-y^\delta+J_k(x_{k+1}-x_{k})]^T[F_k-y^\delta+ J_k(-x_{k}+x^*)] - \frac{2}{\lambda_k}[F_k-y^\delta+J_k(x_{k+1}-x_{k})]^TF_k,
					\end{align*}
					Define $K=\frac{2}{\lambda_k}[F_k-y^\delta+J_k(x_{k+1}-x_{k})]^T[F_k-y^\delta+ J_k(-x_{k}+x^*)]$
					\begin{align*}
						&\norm{x_{k+1}-x^*}_{L}^2-\norm{x_{k}-x^*}_{L}^2\\
						&=-\frac{1}{\lambda_k}(x_{k+1}-x_k)^TJ_k^TF_k -\frac{1}{\lambda_k}\norm{J_k(x_{k+1}-x_k)}^2+K - \frac{2}{\lambda_k}\norm{F_k-y^\delta}^2- \frac{2}{\lambda_k}(x_{k+1}-x_{k})^TJ_k^TF_k \\
						&=-\frac{1}{\lambda_k}[\norm{F_k-y^\delta}^2+2(x_{k+1}-x_{k})^TJ_k^TF_k +\norm{J_k(x_{k+1}-x_k)}^2] -\frac{1}{\lambda_k}(x_{k+1}-x_k)^TJ_k^TF_k\\
						 &\hspace{11cm} -\frac{1}{\lambda_k}\norm{F_k-y^\delta}^2+K\\
						&= -\frac{1}{\lambda_k}\norm{F_k-y^\delta+J_k(x_{k+1}-x_k)}^2-\frac{1}{\lambda_k}(F_k-y^\delta+J_k(x_{k+1}-x_k))^TF_k +K\\
						&= \frac{1}{\lambda_k}\norm{F_k-y^\delta+J_k(x_{k+1}-x_k)}^2-\frac{1}{\lambda_k}(F_k-y^\delta+J_k(x_{k+1}-x_k))^TF_k +K -\beta,
					\end{align*}
					where $\beta= \frac{2}{\lambda_k}\norm{F_k-y^\delta+J_k(x_{k+1}-x_k)}^2$.
					From that, 
					\begin{equation} \label{eq:aux1eqganho}
						\begin{aligned}
							&\norm{x_{k+1}-x^*}_{L}^2-\norm{x_{k}-x^*}_{L}^2\\
							&= \frac{1}{\lambda_k}(F_k-y^\delta+J_k(x_{k+1}-x_k))^T(F_k-y^\delta+ J_k(x_{k+1}-x_k)- F_k+y^\delta) +K -\beta \\
							&= \frac{1}{\lambda_k}(x_{k+1}-x_k)^TJ_k^T(F_k-y^\delta+J_k(x_{k+1}-x_k)) +K -\beta \\
							&= -\norm{L(x_{k+1}-x_k)} +K-\beta \\
							&= -\norm{x_{k+1}-x_k}_{L}  +\frac{2}{\lambda_k}[F_k-y^\delta+J_k(x_{k+1}-x_{k})]^T[F_k-y^\delta+ J_k(-x_{k}+x^*)] \\ &\hspace{8cm}-\frac{2}{\lambda_k}\norm{F_k-y^\delta+J_k(x_{k+1}-x_k)}^2\\
							&\leq -\norm{x_{k+1}-x_k}_{L} +\frac{2}{\lambda_k}\norm{F_k-y^\delta+J_k(x_{k+1}-x_{k})}\norm{F_k-y^\delta+ J_k(-x_{k}+x^*)}\\
							&\hspace{8cm}-\frac{2}{\lambda_k}\norm{F_k-y^\delta+J_k(x_{k+1}-x_k)}^2\\
							&= -\norm{x_{k+1}-x_k}_{L} -\frac{2}{\lambda_k}\norm{F_k-y^\delta+J_k(x_{k+1}-x_{k})}\Big[\norm{F_k-y^\delta+J_k(x_{k+1}-x_k)}\\
							&\hspace{8cm}-\norm{F_k-y^\delta+ J_k(-x_{k}+x^*)}\Big]    
						\end{aligned}
					\end{equation}
					From, \eqref{eq:lemaexistqcondition} and \eqref{eq:qconditionigualdade} we have that
					$$\norm{F_k-y^\delta+J_k(-x_k+x^*)} \leq \frac{1}{\theta}\norm{F_k-y^\delta+J_k(x_{k+1}-x_k)}\leq \norm{F_k-y^\delta+J_k(x_{k+1}-x_k)}.$$
					Therefore, from \eqref{eq:aux1eqganho}
					\begin{equation}
						\norm{x_k - x^*}_{L}^2 - \norm{ x_{k+1} - x^* }_{L}^2 \geq \norm{ x_{k+1} - x_k }^2_{L^TL},    
					\end{equation}
					moreover, from  \eqref{eq:lemaexistqcondition} 
					
					\begin{equation}
						\begin{aligned}\label{eq:aux2eqganho}
							&\norm{x_{k+1}-x^*}_{L}^2-\norm{x_{k}-x^*}_{L}^2\\
							&\leq \frac{2}{\lambda_k}\norm{F_k-y^\delta+J_k(x_{k+1}-x_{k})}\norm{F_k-y^\delta+ J_k(-x_{k}+x^*)}-\frac{2}{\lambda_k}\norm{F_k-y^\delta+J_k(x_{k+1}-x_k)}^2\\
							&\leq \frac{2}{\lambda_k \theta}\norm{F_k-y^\delta+J_k(x_{k+1}-x_{k})}^2 -\frac{2}{\lambda_k}\norm{F_k-y^\delta+J_k(x_{k+1}-x_k)}^2\\
							&= -\frac{2(\theta-1)}{\lambda_k\theta}\norm{F_k-y^\delta+J_k(x_{k+1}-x_k)}^2.
						\end{aligned}
					\end{equation}
					
					From \eqref{eq:qconditionigualdade}, \eqref{eq:intervalolambda} and \eqref{eq:aux2eqganho}
					\begin{equation}
						\begin{aligned}\label{eq:ganho3}
							&\norm{x_{k+1}-x^*}_{L}^2-\norm{x_{k}-x^*}_{L}^2\\
							&\leq -\frac{2(\theta-1)}{\lambda_k\theta}\norm{F_k-y^\delta+J_k(x_{k+1}-x_k)}^2\\
							&=  -\frac{2(\theta-1)}{\lambda_k\theta}q^2\norm{F_k-y^\delta}^2\\
							&\leq -\frac{2(\theta-1)}{(\frac{q}{1-q})\zeta_{p,k}^2\theta}q^2\norm{F_k-y^\delta}^2\\
							&=  -\frac{2(\theta-1)(1-q)q}{\zeta_{p,k}^2\theta}\norm{F_k-y^\delta}^2
						\end{aligned}
					\end{equation}
				\end{proof}

				\begin{teo}\label{teo:converexactdata}
					Let Assumptions~\ref{Hip_nullJLnew}, \ref{hip:tccLTL} and \ref{assumption2.2bellavia2016} hold and $x_k$ be the Levenberg–Marquardt 
					iterates determined by using \eqref{s1} and \eqref{s2}  with exact data. Then, any iterate $x_k$ belongs to $B_{L}(x_0,\rho)$ and  \( \|F_k - y\| \) tends to zero as \( k \to \infty \). %the sequence $\{x_k\}$ converges to a solution of \eqref{probdisc}. 
				\end{teo}
				
				\begin{proof}
					First, we show that $x_k$ belongs to $ B_{L}(x_0, \rho)$. For this, we need \eqref{eq:lemaexistqcondition} to be satisfied, which is the assumption that allows us to use Lemma~\ref{lema:GANHO}.
						  
					Note that, for $k=0$ Assumption
					\ref{hip:tccLTL} implies that 
					\begin{equation}\label{eq:aux1teoexactdata}
						\begin{aligned}
							\norm{F_0-y+J_0(x^*-x_0)}&=\norm{J_0(x^*-x_0) -F(x^*)+F_0}\leq  c\norm{x_0-x^*}_{L}\norm{F_0-y},
						\end{aligned}
					\end{equation}
					 since $F(x^*)=y$.
					 
					If $\norm{x_0-x^*}_{L}=0$, \eqref{eq:aux1teoexactdata} implies that inequality \eqref{eq:lemaexistqcondition} is satisfied for any $\theta > 1$. Otherwise if  $\norm{x_0-x^*}_{L}>0$, Assumption \ref{assumption2.2bellavia2016} implies that					for some $\theta= q (c\norm{x_0-x^*}_{L})^{-1}>1$ inequality \eqref{eq:lemaexistqcondition} is satisfied.
					
					Define \[
					\theta =
					\begin{cases} 
						q \left( c \norm{x_0 - x^*}_{L} \right)^{-1}, & \text{if } \norm{x_0 - x^*}_{L} > 0, \\
						1.1, & \text{otherwise},
					\end{cases}
					\] 
					Then, from Lemma~\ref{lema:GANHO} we have
					$$\norm{x_{1}-x^*}_{L} \leq \norm{x_0-x^*}_{L}.$$ By the induction this inequality remains true for all $k\in \mathbb{N}$ showing
					that the assumptions of Lemma~\ref{lema:GANHO} are satisfied and $\norm{x_k-x^*}_{L}$ is monotonically decreasing.
					
					  Moreover, also by Lemma~\ref{lema:GANHO}, \eqref{eq:plusganho2}, and since $ F' $ is uniformly bounded in $B_{L^T L}(x^*, \rho)$,  we have that \( \|F_k - y\| \) tends to zero as \( k \to \infty \).% which implies that the limit of \( x_k \) must be a solution of \ref{prob1}
				\end{proof}

				\begin{obs}
					Observe that we have obtained an upper bound for the Euclidean norm in terms of the semi norm $L$ 
					\begin{equation}\label{eq:limitsupLTLeuclidiana}
						\begin{aligned}
							\norm{x_{k+1}-x^*} &= \norm{x_k- (J_k^TJ_k+\lambda_kL^TL)^{-1}J_k^T(F_k-y)-x^*}\\ &=\norm{ (J_k^TJ_k+\lambda_kL^TL)^{-1}(J_k^T(F_k-y)+(J_k^TJ_k+\lambda_kL^TL)(x_k-x^*))}\\&\leq \norm{ (J_k^TJ_k+\lambda_kL^TL)^{-1}}\norm{(J_k^T(F_k-y)+(J_k^TJ_k+\lambda_kL^TL)(x_k-x^*))}\\&= \norm{ (J_k^TJ_k+\lambda_kL^TL)^{-1}}\norm{(J_k^T(F_k-y+J_k(x_k-x^*))+\lambda_kL^TL(x_k-x^*))}	
							\\&\leq \norm{ (J_k^TJ_k+\lambda_kL^TL)^{-1}}(\norm{J_k^T}\norm{F_k-y+J_k(x_k-x^*)}+\lambda_k\norm{L^T}\norm{x_k-x^*}_{L}) \\
							&\leq \norm{ (J_k^TJ_k+\lambda_kL^TL)^{-1}}(\norm{J_k^T}c\norm{F_k-y}\norm{x_k-x^*}_{L}+\lambda_k\norm{L^T}\norm{x_k-x^*}_{L})
						\end{aligned}
					\end{equation}	
				with $c$ as in \eqref{eq:tcc}. 
				\end{obs}

				\begin{teo}\label{teo:Convergenciaxkexato}			
				Let Assumptions~\ref{Hip_nullJLnew}, \ref{hip:tccLTL} and \ref{assumption2.2bellavia2016} hold and $x_k$ be the Levenberg–Marquardt 
				iterates determined by using \eqref{s1} and \eqref{s2}  with exact data. Then, the sequence $\{x_k\}$ converges to a solution of \eqref{probdisc}. 
			\end{teo}

			\begin{proof}
				Without loss of generality, assume that $x_k - x^* \notin N(L)$, otherwise, from \eqref{eq:limitsupLTLeuclidiana}, we would have $x_{k+1} = x^*$, and there would be nothing to prove.

				From Assumption~\ref{Hip_nullJLnew} for $v=(x_k-x^*)$ we have
				\begin{equation}\label{eq:hipnullforxkx*}
					 \gamma\norm{x_k-x^*}^2 \leq \norm{J_k(x_k-x^*)}^2+\norm{L(x_k-x^*)}^2
				\end{equation}
				
				In the proof of Theorem~\ref{teo:converexactdata} we show that $\norm{x_k-x^*}_{L}$ is monotonically decreasing, then $\norm{x_k-x^*}_{L}$ is convergent, moreover $\norm{F_k-y} \rightarrow 0$ as $k\rightarrow \infty$. 
				
				From \eqref{eq:lemaexistqcondition} 
				$$\norm{F_k-y}-\norm{J_k(x^*-x_k)} \leq \frac{q}{\theta}\norm{F_k+-y}$$
				and 
				$$-\norm{F_k-y}+\norm{J_k(x^*-x_k)} \leq \frac{q}{\theta}\norm{F_k-y}.$$
				
				Then, 
				$$ (1-\frac{q}{\theta})\norm{F_k-y}\leq\norm{J_k(x^*-x_k)} \leq (1+\frac{q}{\theta})\norm{F_k-y},$$
				and $\norm{J_k(x_k-x^*)}$ tends to zero as $k\rightarrow \infty$.

				From \eqref{eq:hipnullforxkx*}, we conclude that $\{x_k\}$ is bounded. Therefore, $\{x_k\}$ has a convergent subsequence. Let $\bar{x}$ be the limit of this subsequence  
				for $k \in \mathcal{K}$. Since $F$ is continuous and $F(x_k) \to y$, it follows that $F(\bar{x}) = y$.  
				
				Let $\varepsilon > 0$. From \eqref{eq:hipnullforxkx*} and \eqref{eq:ganho}, for $k > \bar{k} \in \mathcal{K}$ such that $\norm{x_{\bar{k}} - \bar{x}} \leq \frac{\gamma \varepsilon^2}{2\norm{L}^2}$ and $\norm{J_k(\bar{x} - x_k)}^2 \leq \frac{\gamma \varepsilon}{2}$, we obtain  
				\begin{equation}
					\begin{aligned}
						\gamma\norm{x_k - \bar{x}}^2 &\leq \norm{J_k(x_k - \bar{x})}^2 + \norm{L(x_k - \bar{x})}^2 \\  
						&\leq \norm{J_k(x_k - \bar{x})}^2 + \norm{x_{\bar{k}} - \bar{x}}_{L}^2 - \sum_{l = \bar{k}}^k \norm{x_{l+1} - \bar{x}}_{L}^2 \\  
						&\leq \norm{J_k(x_k - \bar{x})}^2 + \norm{L}^2 \norm{x_{\bar{k}} - \bar{x}}^2 \\  
						&\leq \frac{\gamma \varepsilon}{2} + \frac{\norm{L}^2 \gamma \varepsilon^2}{2\norm{L}^2} \\  
						&= \gamma \varepsilon^2.
					\end{aligned}
				\end{equation}
				Then, $\norm{x_k - \bar{x}} \leq \varepsilon$ and we conclude that $\{x_k\}$ converges to a solution of \eqref{probdisc}.
			\end{proof}
		
		\begin{teo}\label{teo:convergenciacomruido}
			Let Assumptions~\ref{Hip_nullJLnew}, \ref{hip:tccLTL} and \ref{assumption2.2bellavia2016} hold and $x_k$ be the Levenberg–Marquardt
			iterates determined by using \eqref{s1} and \eqref{s2}. For noisy data, suppose $k \leq k^*(\delta)$where $k^*(\delta)$ is
			defined in \eqref{criteriodiscrepancia}. Then, any iterate $x_k$ belongs to $B_{L}(x_0,\rho)$. Moreover the stopping criterion \eqref{criteriodiscrepancia} is
			satisfied after a finite number $k^*(\delta)$ of iterations and  $\{x_{k^*(\delta)}^{\delta}\}$   converges to a solution
			of \eqref{probdisc} as $\delta$ tends to zero.	
		\end{teo}
		\begin{proof}
			
			First of all, we show monotonicity of $\norm{x_k^\delta - x^*}_{L}$ up to the stopping index defined by \eqref{criteriodiscrepancia}.  If $k^*(\delta) = 0$, nothing has to be shown. Let us now assume that $k^*(\delta) \geq 1$, then for $k=0$ note that Assumption
			\ref{hip:tccLTL} implies that 
			
			\begin{equation}\label{eq:aux1teotaudeltadata}
				\begin{aligned}
					\norm{F_0-y^\delta+J_k(x^*-x_0^\delta)}&=\norm{J_0(x^*-x_0^\delta) -F(x^*)+y+F_0-y^\delta}\\&\leq c\norm{x_0^\delta-x^*}_{L}\norm{F_0-F(x^*)}+\norm{y-y^\delta}\\ &\leq c\norm{x_0^\delta-x^*}_{L}\norm{F_0-y}+\delta\\
					&\leq c\norm{x_0^\delta-x^*}_{L}\norm{F_0-y^\delta+y^\delta-y}+\delta\\
					&\leq c\norm{x_0^\delta-x^*}_{L}\norm{F_0-y^\delta}+(1+c\norm{x_0^\delta-x^*}_{L})\delta \\
					&\leq c\norm{x_0^\delta-x^*}_{L}\norm{F_0-y^\delta} +(1+c\norm{x_0^\delta-x^*}_{L})\frac{\norm{F_0-y^\delta}}{\tau}\\
					&= \frac{1+c(1+\tau)\norm{x_0^\delta-x^*}_{L}}{\tau}\norm{F_0-y^\delta}
				\end{aligned}
			\end{equation}
			If $\norm{x_0^\delta-x^*}_{L}=0$, \eqref{eq:aux1teotaudeltadata} and Assumption~\ref{assumption2.2bellavia2016} implies that inequality \eqref{eq:lemaexistqcondition} is satisfied for any $\theta \geq q\tau>1$. Otherwise if  $\norm{x_0^\delta-x^*}_{L}>0$, Assumption \ref{assumption2.2bellavia2016} implies that					for some $\theta=\frac{q\tau}{1+c(1+\tau)\norm{x_0^\delta-x^*}_{L}}>1$ inequality \eqref{eq:lemaexistqcondition} is satisfied.
			Define \[
			\theta =
			\begin{cases} 
				\frac{q\tau}{1+c(1+\tau)\norm{x_0^\delta-x^*}_{L}}, & \text{if } \norm{x_0^\delta - x^*}_{L} > 0, \\
				q\tau, & \text{otherwise},
			\end{cases}
			\] 
			Then, from Lemma~\ref{lema:GANHO} we have
			$$\norm{x_{1}^\delta-x^*}_{L} \leq \norm{x_0^\delta-x^*}_{L}.$$ By the induction this inequality remains true for all $k\in \mathbb{N}$ showing
			that the assumptions of Lemma~\ref{lema:GANHO} are satisfied and $\norm{x_k^\delta-x^*}_{L}$ is monotonically decreasing.
			
			Summing up both sides of \eqref{eq:plusganho2} from $0$ to $k^* -1$, one obtains with \eqref{criteriodiscrepancia} that 
			$$k^*\tau^2\delta^2 \leq \sum^{k^*-1}_{k=0}\norm{y^\delta-F_k}^2 \leq \frac{\theta \hat{\zeta}^2}{2(\theta-1)(1-q)q}\norm{x_0-x^*}_{L}, $$
			where
			$\hat{\zeta}=\max{\zeta_{p,i}}$ for ${0 \leq i\leq k^*-1}.$
			Thus, $k^*$ is finite for $\delta>0$.
			
			Let $\bar{x}$ be the limit of the sequence \{$x_k$\} corresponding to the exact data $y$ (see Theorem~\ref{teo:Convergenciaxkexato}) and let \{$\delta_n$\} be a sequence of values of $\delta$ converging to zero as $n \to \infty$. Denote by $y^{\delta_n}$ a corresponding sequence of perturbed data, and by $k^*(\delta_n)$ the stopping index determined from the discrepancy principle \eqref{criteriodiscrepancia} applied with $\delta=\delta_n$. We now show the convergence in the two possible cases.
			
			Case 1: Assume that $\tilde{k}$ is a finite accumulation point of $\{k^*(\delta_n)\}$. Without loss of generality, for the monotonicity of \eqref{prob1}, we can assume that $k^*(\delta_n) = \tilde{k}$ for all $n \in \mathbb{N}$. Thus, from the definition of $k^*(\delta_n)$ it follows that 
			\begin{equation}\label{eq:2.20kalt}
				\norm{y^{\delta_n}- F(x_{\tilde{k}}^{\delta_n})}\leq \tau \delta_n.
			\end{equation}			
			
			Then, $\norm{y^{\delta_n}- F(x_{\tilde{k}}^{\delta_n})}$ tends to zero as $\delta_n \rightarrow 0$. 
			As $\tilde{k}$ is fixed, and $x_{\tilde{k}}^{\delta_n}$ depends continuously on $y^\delta$ (see Remark~\ref{obs:xkcontydelta}), we obtain
			$$x_{\tilde{k}}^{\delta_n} \rightarrow x_{\tilde{k}}, F(x_{\tilde{k}}^{\delta_n})\rightarrow F(x_{\tilde{k}}) \text{ as } n\rightarrow \infty.$$
			
			From \eqref{eq:2.20kalt} we have that $F(x_{\tilde{k}}^{\delta_n})=y^{\delta_n}$ when $\delta_n\rightarrow 0$ and the $\tilde{k}$-th iterate with exact data $y$ is a solution of $F(x) = y= F(\bar{x})$ from Theorem~\ref{teo:Convergenciaxkexato} (note that in the exact data for this case $x_{\tilde{k}+1}=x_{\tilde{k}}$). Then, we can conclude that $x_{\tilde{k}}^{\delta_n} \rightarrow x_{\tilde{k}}=\bar{x}.$
			Then $x_{\tilde{k}}^{\delta_{n}} \to \bar{x}$ as $\delta_n \to 0$.

			Case 2: Considere where $k^*(\delta_n) \to \infty$ as $n \to \infty$. 	
				From Assumption~\ref{Hip_nullJLnew} for $v=(x_{k^*(\delta_n)}^{\delta_n}-\bar{x})$ we have
			\begin{equation}\label{eq:hipnullforxkx*2}
				\gamma\norm{x_{k^*(\delta_n)}^{\delta_n}-\bar{x}}^2 \leq \norm{J_{k^*(\delta_n)}(x_{k^*(\delta_n)}^{\delta_n}-\bar{x})}^2+\norm{L(x_{k^*(\delta_n)}^{\delta_n}-\bar{x})}^2=\norm{J_{k_n}(x_{k^*(\delta_n)}^{\delta_n}-\bar{x})}^2+\norm{x_{k^*(\delta_n)}^{\delta_n}-\bar{x}}_{L}^2
			\end{equation}
			
			For $k^*(\delta_n)>k_n>0$, \eqref{eq:ganho} implies that
			
			\begin{equation}\label{eq:auxcase2withdisc2}
				\norm{x_{k^*(\delta_n)}^{\delta_n}-\bar{x}}_{L}\leq  \norm{x_{k_n}^{\delta_n}-\bar{x}}_{L}\leq \norm{x_{k_n}^{\delta_n}-x_k}_{L} + \norm{x_k-\bar{x}}_{L}
			\end{equation}
			
			where $x_k$ is the iterate with exact data. 
			From Theorem~\ref{teo:Convergenciaxkexato}, exist $\varepsilon>0$ (for for sufficiently large $k$) such that $\norm{x_k-\bar{x}}\leq \frac{1}{2\norm{L}}\varepsilon$. As $x_k^\delta$ depends continuously on the data $y^\delta$ and $x_k$ depends continuously on the $y$, then for $n \rightarrow \infty$, $\norm{x_{k_n}^{\delta_n}-x_k}\leq \frac{1}{2\norm{L}}\varepsilon.$
			From, \eqref{eq:auxcase2withdisc2} we have
			$$\norm{x_{k^*(\delta_n)}^{\delta_n}-\bar{x}}_{L}\leq \norm{L}\norm{x_{k_n}^{\delta_n}-x_k} + \norm{L}\norm{x_k-\bar{x}}\leq \varepsilon.$$
			 
			Moreover, from \eqref{eq:lemaexistqcondition} 
			$$ (1-\frac{q}{\theta})\norm{F_{k^*(\delta_n)}-y^{\delta_n}}\leq\norm{J_{k^*(\delta_n)}(x_{k^*(\delta_n)}-\bar{x})} \leq (1+\frac{q}{\theta})\norm{F_{k^*(\delta_n)}-y^{\delta_n}},$$
			and $\norm{J_{k^*(\delta_n)}(x_{k^*(\delta_n)}-\bar{x})}$ tends to zero as $n\rightarrow \infty$.
			
			We conclude from \eqref{eq:hipnullforxkx*2} that 
			$\norm{x_{k^*(\delta_n)}^{\delta_n}-\bar{x}}$ converge for zero when $n$ tends to infinity and $\delta$ tends to zero.
			
		\end{proof}

           %		\section{Numerical}\label{sec:expnum}

				\section{Final remarks}\label{sec:conclusion}
				
				In this work, we proposed and analyzed a regularization for the Levenberg–Marquardt method with Singular Scaling (LMMSS) applied to nonlinear inverse problems with noisy data. By allowing the use of a possibly singular scaling matrix $L$, the method extends the classical Levenberg–Marquardt approach and provides greater flexibility in the regularization process.
				
				We introduced a variant of the classical Tangent Cone Condition adapted to the $L$-induced geometry, which was crucial for establishing the regularizing properties of the method. Under this assumption, we proved that the iterates produced by the LMMSS method converge to a solution of the unperturbed problem as the noise level tends to zero.
				
				A key contribution of this work is the derivation of an inequality that estimates the decrease in the $L$-norm distance to the exact solution at each iteration. This result played a central role in the convergence analysis, both in the exact and noisy data settings. Additionally, we showed that the regularization parameter $\lambda_k^q$ implicitly defined via a $q$-condition—varies continuously with the data, ensuring the stability of the method with respect to perturbations.
				
				%\tred{PARAGRAFO SOBRE OS EXPERIMENTOS}
				
				%Although a theoretical analysis is not yet provided, we shall explore the validity of the TCC-L assumption in specific applications and investigate adaptive strategies for selecting the parameter $q$ through numerical experiments.
				A theoretical investigation of the TCC-L assumption in concrete applications, as well as the development of adaptive strategies for selecting the parameter 
				$q$, will be pursued in future work, potentially supported by numerical experiments.
				
				This study represents a first step toward the theoretical foundation of the LMMSS method as a regularization tool. Future work may include the extension of the analysis to the case of nonzero-residual problems, as well as numerical experiments to assess the practical performance of the method in inverse problems.
				
				\section*{Acknowledgments}
				This work was supported by Brazilian agencies FAPESC (Fundação de Amparo à Pesquisa e Inovação do Estado de Santa Catarina) [grant number 2023TR000360] and CNPq (Conselho Nacional de Desenvolvimento Científico e Tecnológico). DG ackownledges the support of CNPq (Conselho Nacional de Desenvolvimento Científico e Tecnológico), Brazil [grant number 305213/2021-0].

				\bibliographystyle{abbrv}
				\bibliography{refs}
				
			\end{document}